\documentclass[12pt,a4paper]{amsart}
\usepackage{amssymb,color,currfile,pdfpages}

\usepackage[english]{babel}
\usepackage{verbatim,here}
\usepackage[T1]{fontenc}
\usepackage{epstopdf}
\usepackage{floatflt,graphicx}
\usepackage{color,xcolor}
\usepackage{enumerate}
\usepackage{animate}
\usepackage{a4wide}
\usepackage{amsmath,amssymb}
\usepackage{amsthm}
\usepackage{float}

\usepackage{orcidlink}
\usepackage{hyperref}

\usepackage{tikz}
\usetikzlibrary{arrows.meta,3d}
\usepackage{subcaption}
\usepackage[english]{babel}
\usepackage{verbatim,here}
\usepackage[T1]{fontenc}

\usepackage{epstopdf}
\usepackage{floatflt,graphicx,framed}
\usepackage{color,xcolor}
\usepackage{enumerate}
\usepackage{animate}
\usepackage{a4wide}

\usepackage{float}

\newcounter{minutes}
\divide\time by 60
\newcounter{hours}
\multiply\time by 60 \addtocounter{minutes}{-\time}

\dedicatory{}
\commby{}
\theoremstyle{plain}
\newtheorem{thm}[equation]{Theorem}
\newtheorem{cor}[equation]{Corollary}
\newtheorem{lem}[equation]{Lemma}

\newtheorem{prop}[equation]{Proposition}

\theoremstyle{definition}

\theoremstyle{remark}

\newtheorem{nonsec}[equation]{}

\numberwithin{equation}{section}

\newcommand{\beq}{\begin{equation}}
\newcommand{\eeq}{\end{equation}}
\newcommand{\ben}{\begin{enumerate}}
\newcommand{\een}{\end{enumerate}}
\newcommand{\bequu}{\begin{eqnarray*}}
\newcommand{\eequu}{\end{eqnarray*}}
\newcommand{\bequ}{\begin{eqnarray}}
\newcommand{\eequ}{\end{eqnarray}}

\begin{document}
\thispagestyle{empty}
\def\thefootnote{}

\title[The Ptolemy-Alhazen problem]{The Ptolemy-Alhazen problem with source
at infinity}

\author[M. Fujimura]{Masayo Fujimura \orcidlink{0000-0002-5837-8167}}
\address{Department of Mathematics,
National Defense Academy of Japan, Yokosuka, Japan\newline
\href{https://orcid.org/0000-0002-5837-8167}{{\tt https://orcid.org/0000-0002-5837-8167}}
}
\email{masayo@nda.ac.jp}%
\author[M. Vuorinen]{Matti Vuorinen \orcidlink{0000-0002-1734-8228}}
\address{Department of Mathematics and Statistics,
University of Turku, Finland\newline
\href{https://orcid.org/0000-0002-1734-8228}{{\tt https://orcid.org/0000-0002-1734-8228}}
}
\email{vuorinen@utu.fi}
\date{}

\begin{abstract}
We study the well-known Ptolemy-Alhazen problem on reflection of light
at the surface of a spherical mirror in the case when the source of light
is very far from the mirror.
\end{abstract}


\maketitle

\footnotetext{\texttt{{\tiny File:~\jobname .tex, printed: \number\year-%
\number\month-\number\day, \thehours.\ifnum\theminutes<10{0}\fi\theminutes}}}
\makeatletter

\makeatother



\section{Introduction}

The well-known
Ptolemy--Alhazen problem reads \cite{s},\cite[p.1010]{gbl}: "Given a light source and a spherical mirror, find the point on the mirror where the light 
will be reflected to the eye of an observer." The long history of this problem is described in \cite{s,fmv2,w}.  The problem is of current interest for instance in the design of
computer vision of drones, in the study of GPS signal transmission,  in many scattering
problems of electromagnetic signals (e.g. radar signals), acoustics, and in various ray-tracing problems  of
astrophysics \cite{atr,dsggh,le, m, dg1,dg2}.  We have proved that this problem is equivalent
to the question of finding a formula for the so called triangular ratio metric of geometric function theory
and also studied the problem in the case of quartic mirrors \cite{fhmv,fmv1, fmv2,hkv}.

Since the publication of our work \cite{fhmv}, our solution has been implemented and validated in astrophysics applications in \cite{dg1, dg2} and its performance compared to other methods.  V.H.  Almeida Jr., one of the authors of \cite{dg1,dg2},
has sent us an email \cite{d} where he pointed out the need for the solution of the problem in the case of the source at
infinity. A few days after the receipt of his email we informed him about the solution of this
problem and gave the solution in Theorem \ref{claim1} below. The solution
is given by a fourth degree polynomial equation which one can solve with symbolic
computation.

 \section{Preliminaries}

For the reader's benefit we recall here some of our earlier
work \cite{fhmv, hkv}.

The triangular
ratio metric $s_G$ of a given domain $G \subset {\mathbb{R}}^2$ is defined as
follows for $z_1,z_2 \in G$ \cite{hkv}
\begin{equation*}  
  s_G(z_1,z_2)= \sup_{z \in \partial G} \frac{|z_1-z_2|}{|z_1-z|+|z-z_2|} \, .
\end{equation*}
By compactness, this supremum is attained at some point $w \in \partial
G\,.$ If $G$ is convex, it is simple to see that $w$ is the point of
contact of the boundary with an ellipse, with foci $z_1,z_2\,,$ contained in
$G\,.$ Now for the case $G= \mathbb{B}^2$ and $z_1,z_2 \in \mathbb{B}^2\,$, if the extremal point is $w
\in \partial \mathbb{B}^2\,,$  then
\begin{equation} \label{extrpt}
  s_{\mathbb{B}^2}(z_1, z_2) = \frac{|z_1- z_2|}{ |z_1- w| + |z_2-w| } \,.
\end{equation}
This definition also makes sense for the domain $\mathbb{R}^2 \setminus \overline{\mathbb{B}}^2,$ we call the cases 
$\mathbb{B}^2$ and $\mathbb{R}^2 \setminus \overline{\mathbb{B}}^2$ {\it the interior and exterior problems,} resp. 
The connection between the triangular ratio
distance
and the Ptolemy-Alhazen interior problem is now clear: if the unit circle is a reflecting mirror and $z_1$ is the light source and $z_2$ is the observer, then the point $w$ is the point of reflection.
We also see that an ellipse with foci $z_1$ and $z_2$ through
the point $w$ is internally tangent to the unit circle.

\begin{nonsec}{\bf Ellipse eccentricity.}\label{excent}
Let $z_1,z_2 \in \mathbb{B}^2 $ and $w \in \partial \mathbb{B}^2$ be the extremal point of \eqref{extrpt}. Then the focal sum of the ellipse is
\[
c=|z_1-w|+|z_2-w|=|z_1-z_2|/s_{\mathbb{B}^2}(z_1,z_2) \,.
\]

Let $\lambda, \sigma$ be the lengths of the  major and minor semiaxes of the ellipse. Then
\[
\lambda= c/2\,,\quad \sigma= \frac{1}{2} \sqrt{c^2 - |z_1-z_2|^2}
\]
and the eccentricity equals
\[
\sqrt{\frac{\lambda^2- \sigma^2}{\lambda^2}}= \sqrt{1- \left( \frac{\sigma}{\lambda}\right)^2} = s_{\mathbb{B}^2}(z_1,z_2) \,.
\]
In conclusion, $ s_{\mathbb{B}^2}(z_1,z_2)$ is equal to the eccentricity of the maximal ellipse in $ {\mathbb{B}^2}$ with foci $z_1,z_2\,.$
\begin{figure}[H] %
\centerline{
\includegraphics[width=0.4\linewidth]{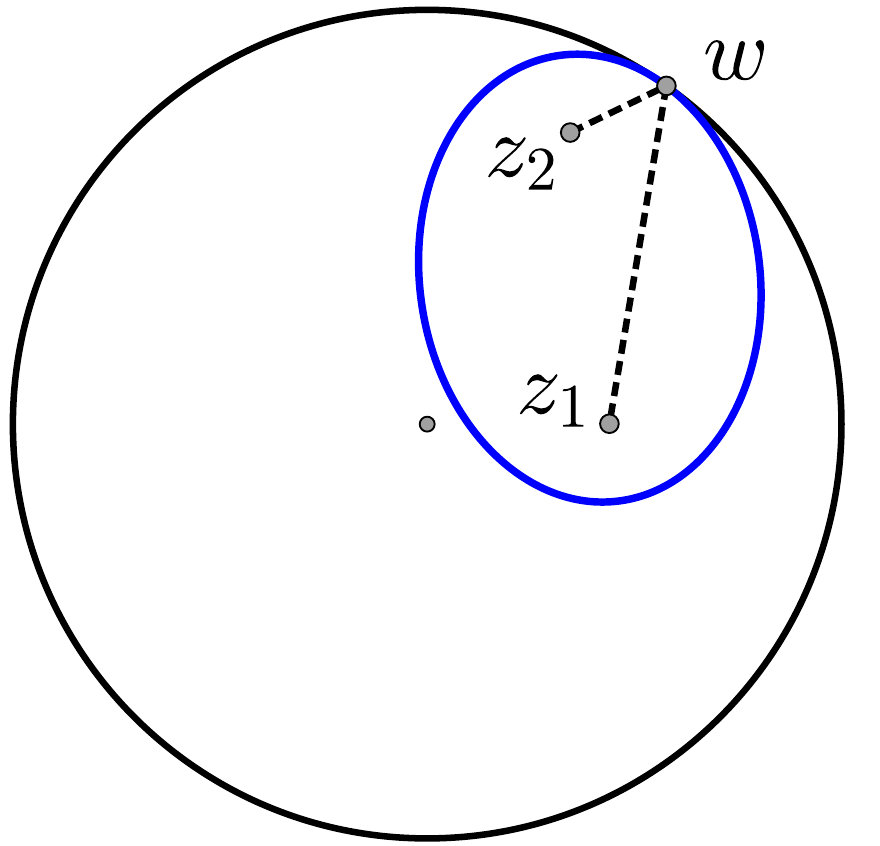}
}
\caption{If the unit circle is a mirror, then a light signal from the source $z_1$ to the observer at $z_2$ is reflected at the point $w.$ Moreover, $ s_{\mathbb{B}^2}(z_1,z_2)$ is equal to the eccentricity of the maximal ellipse.} \label{figEllip}
\end{figure}

\end{nonsec}

The main result of \cite{fhmv},
Theorem \ref{thm:Fuji}, gives a polynomial equation of degree four that yields the
reflection point on the unit circle.
The solution of the quartic equation is given by Cardano's formula which is
available in standard symbolic computation software.  We also study the Ptolemy-Alhazen exterior
problem for $z_1, z_2 \in \mathbb{R}^2 \setminus \overline{\mathbb{B}}^2,$ and in this case the formula for the reflection point, if such a point exists,
is the same as for the interior problem \cite[Rmk 2.7]{fhmv}. 

\begin{thm}\label{thm:Fuji}
For $z_1, z_2 \in {\mathbb{B}}^2,$  the point $w $ in \eqref{extrpt} is given as a
  solution of the  equation
  \begin{equation}  \label{eq:equation}
     \overline{z_1}\overline{z_2}w^4-(\overline{z_1}+\overline{z_2})w^3
     +(z_1+z_2)w-z_1z_2=0\,.
  \end{equation}
\end{thm}

\begin{nonsec}{\bf Remark.} The roots of equation
 \eqref{eq:equation} have been studied in detail
in \cite{fhmv}.  The root $w$ at the extremal point  for  \eqref{extrpt} is called the minimizing root. For instance, it is known that the minimizing root need not be unique and that some of the roots of  \eqref{eq:equation} may be off $ \partial {\mathbb{B}^2}.$
A computer algorithm for finding
the minimizing root  and a formula for $ s_{\mathbb{B}^2}(z_1, z_2)$ based on this root is given in   \cite{fhmv}.
\end{nonsec}

\begin{cor}\label{cor:sDformula}
  For $z_1, z_2 \in \mathbb{B}^2,$ we have
  \begin{equation*}
    s_{\mathbb{B}^2}(z_1, z_2) = \frac{|z_1- z_2|}{|z_1-w|+ |z_2-w| }
  \end{equation*}
  where $w \in \partial {\mathbb{B}^2 } $ is the minimizing root of
  \eqref{eq:equation}\,.
\end{cor}

\begin{nonsec}{\bf Testing and validation.}\label{tstvali}
The above method has been applied in astrophysics applications.
Some of the comments read as follows.

(1) From \cite[p.7]{dg1}:
"Furthermore, we compared algorithms at $45^o$ to Fujimura et al.'s
algorithm. We found Mart\'in--Neira and Helm had a systematic discrepancy in terms of grazing angle and arc length, proportional to antenna
heights. Fermat's iterative algorithm presented the least accurate results. Therefore, we advise users to give preference to Fujimura et al.'s
algorithm."

(2) From  \cite[p.3304]{dg2}: "Several formulations were previously tested and we found the model of Fujimura et al.
(2019) to be more numerically stable (Almeida Junior
and Geremia--Nievinski, 2023)."

(3) From \cite{mbm}: "Fujimura et al. (2019) provides a formulation that works well for the method of path--length minimization, but suffers numerical instability in cases where the radius of the sphere is very small relative to source or observer distance (e.g., in the one--finite case discussed below). Glaeser (1999) demonstrates the numerical root--finding approach (using an algorithm from Schwarze 1990) and notes that it is subject to "numerical instabilities that may lead to a considerable loss of accuracy" under certain circumstances."  

 \end{nonsec}
\vspace{3cm}

\section{Quartic equation}


Assume that light rays reach the unit circle from the right side, 
parallel to the real axis.

\begin{figure}[H] %
\centerline{
\includegraphics[width=0.5\linewidth]{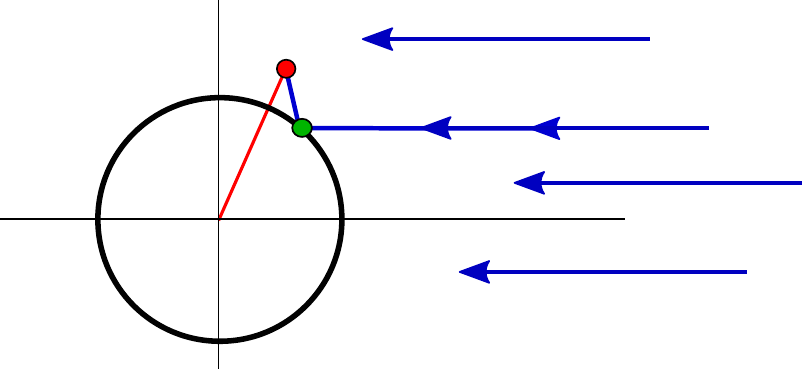}
}
\caption{Light rays reach the unit circle from the right side.}
\label{figTab}
\end{figure}

Let $ f =re^{i\theta} $ be the observation point,
where we can assume $ r>1 $ and $ 0\leq \theta\leq \frac{\pi}2 $
from symmetry.

Let $ w $ be the reflection point and $ \widetilde{w} $
with $ \widetilde{w}=w+1 $.
Note that $ w $ and $ \widetilde{w} $ are points on a line 
parallel to the real axis.

\begin{figure}[H] %
\centerline{
\includegraphics[width=0.4\linewidth]{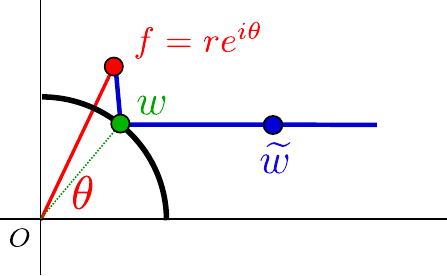}
}
\caption{The light ray reflects at the point $w$ on the unit circle
 and reaches the point $f$.} \label{figTab2}
\end{figure}

\begin{thm}
\label{claim1}
The reflection point $ w $ is given as a solution of the equation
\begin{equation}\label{eq:1}
   re^{-i\theta}w^4-w^3+w-re^{i\theta}=0.
\end{equation}
\end{thm}

\noindent
\begin{proof} \quad From the reflection property, the following holds
$$ \measuredangle(0,w,\widetilde{w})=\measuredangle(f,w,0). $$
Therefore, 
$  \displaystyle \arg \dfrac{\widetilde{w}-w}{0-w} = \arg \dfrac{0-w}{f-w} $
holds and
\begin{align*}
 \arg \dfrac{\widetilde{w}-w}{-w} - \arg \dfrac{-w}{f-w}
  &  = \arg \dfrac{w-\widetilde{w}}{w} \cdot \dfrac{w-f}{w} \\
  & = \arg \dfrac{(w-(1+w))(w-f)}{w^2}
   = \arg \dfrac{f-w}{w^2}=0.
\end{align*}
The last equality implies that $\dfrac{f-w}{w^2} $ is a real number. 
Since this complex conjugate is also real, we have
$$
    \frac{f-w}{w^2}=\frac{\overline{f}-\overline{w}}{\overline{w}^2}.
$$
Hence, we have
$$
  \overline{w}^2(re^{i\theta}-w)=w^2(re^{-i\theta}-\overline{w}).
$$
Since $ w $ is a point on the unit circle and satisfies $w\overline{w}=1$, 
the above equation can be written as
$$
   re^{-i\theta}w^4-w^3+w-re^{i\theta}=0.
$$
\end{proof}

\noindent
\begin{nonsec}{\bf Remark.}\label{rmk2} Cardano's formula gives
the four roots of equation \eqref{eq:1}.
Under the assumptions here, we can choose the root that can 
be written as $ e^{i\varphi} $ ($ 0\leq \varphi \leq \frac{\pi}2 $).
\end{nonsec}

The point $w$ is the point of tangency of the unit circle to 
the parabola whose directrix is perpendicular to the real axis 
and whose focus is $f$.


\begin{figure}[H] %
\centerline{
\includegraphics[width=0.5\linewidth]{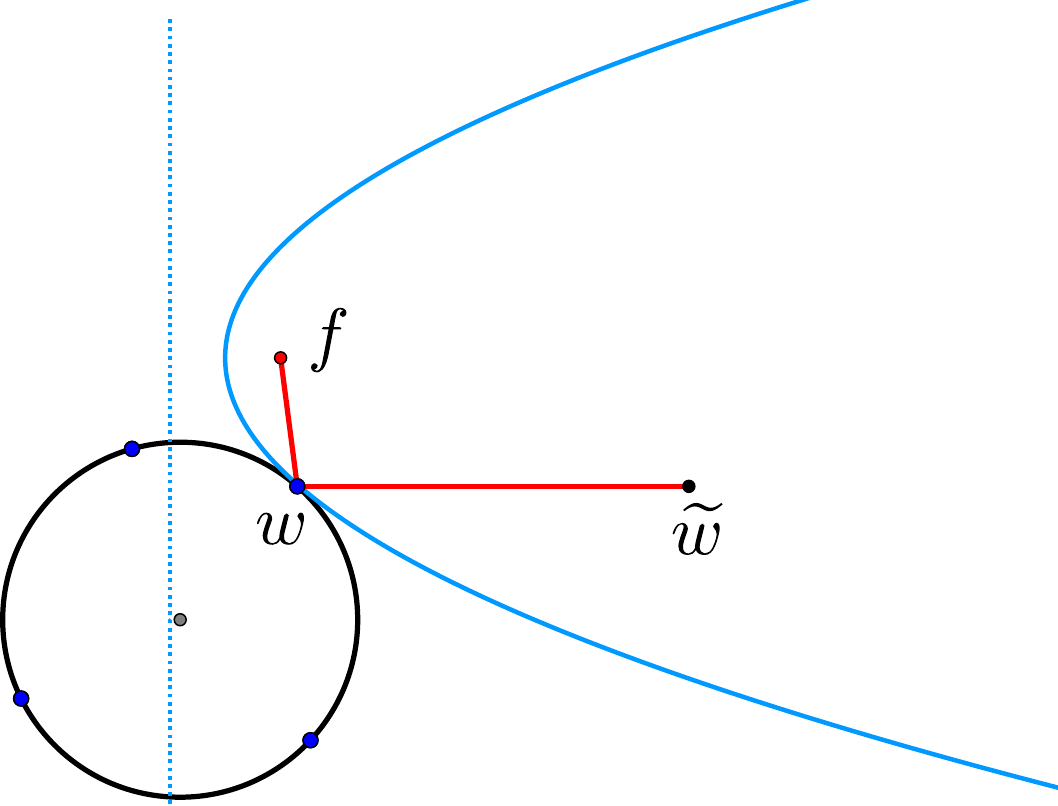}
}
\caption{Tangential parabola and the four roots of
a quartic equation on the unit circle. See Theorem \ref{claim2}.} \label{figTab4}
\end{figure}

\begin{thm}\label{claim2}
For $ r>1 $, the four roots of the equation 
\begin{equation}\label{eq:inf}
  re^{-i\theta}w^4-w^3+w-re^{i\theta}=0
\end{equation}
lie on the unit circle.
\end{thm}

\noindent \begin{proof}
If \eqref{eq:inf} has $ w=1 $ as its solution, 
then $ \theta=0 $ follows from $ r(e^{-i\theta}-e^{i\theta})=0 $.
In this case, the corresponding parabola degenerates to a 
half-line on the real-axis because the light ray arrives 
directly above the observer.
So from now on we will assume that \eqref{eq:inf}
does not have $w=1$ as its solution.

Consider the following two equations
\begin{equation}\label{eq:D}
  (w-w_1)(w-w_2)(w-w_3)(w-w_4)=0
\end{equation}
and
\begin{equation}\label{eq:R}
  (z-z_1)(z-z_2)(z-z_3)(z-z_4)=0,
\end{equation}
where $ z_k=i\dfrac{1+w_k}{1-w_k} $ ($ k=1,2,3,4 $).
We remark that, for $ k=1,\cdots,4 $,  $ w_k\in\partial \mathbb{B}^2 $
and $ z_k\in\mathbb{R} $ are equivalent.

Suppose that $ w_1,\cdots,w_4 $ are the solutions to \eqref{eq:inf}.
Then, from Vieta's formulas, the following equalities hold
\begin{align} \notag
  &  (w_1+w_2+w_3+w_4)r-e^{i\theta}=0,\\  \notag
  &  w_1(w_2+w_3+w_4)+w_2(w_3+w_4)+w_3w_4=0, \\ \notag
  &  (w_1w_2w_3+w_1w_2w_4+w_1w_3w_4+w_2w_3w_4)r+e^{i\theta}=0, \\ \label{eq:s}
  & w_1w_2w_3w_4+e^{2i\theta}=0.
\end{align}
The equation \eqref{eq:R} can be written as 
\begin{align}\notag
  &(w_1-1)(w_2-1)(w_3-1)(w_4-1)z^4 \\ \notag
  & \quad +2i\big((w_1+w_2+w_3+w_4)-(w_1w_2w_3+w_1w_2w_4+w_1w_3w_4+w_2w_3w_4)\\
     \notag
  & \qquad
        +2w_1w_2w_3w_4-2\big)z^3 \\ \notag
  & \quad +2\big((w_1(w_2+w_3+w_4)+w_2(w_3+w_4)+w_3w_4)-3w_1w_2w_3w_4-3\big)z^2 \\
       \notag
  & \quad +2i\big((w_1+w_2+w_3+w_4)+(w_1w_2w_3+w_1w_2w_4+w_1w_3w_4+w_2w_3w_4)\\
    \notag
  & \qquad
           -2w_1w_2w_3w_4+2\big)z \\  \label{eq:R2}
  & \quad +(w_1+1)(w_2+1)(w_3+1)(w_4+1)=0
\end{align}
and the following hold,
\begin{align*}
 & (w_1-1)(w_2-1)(w_3-1)(w_4-1) \\
 & \quad =1-(w_1+w_2+w_3+w_4)
    +(w_1(w_2+w_3+w_4)+w_2(w_3+w_4)+w_3w_4) \\
 & \qquad   -(w_1w_2w_3+w_1w_2w_4+w_1w_3w_4+w_2w_3w_4)+w_1w_2w_3w_4, \\
 & (w_1+1)(w_2+1)(w_3+1)(w_4+1) \\ 
 & \quad  =1+(w_1+w_2+w_3+w_4)
      +(w_1(w_2+w_3+w_4)+w_2(w_3+w_4)+w_3w_4) \\
 & \qquad   +(w_1w_2w_3+w_1w_2w_4+w_1w_3w_4+w_2w_3w_4)+w_1w_2w_3w_4.
\end{align*}
Substituting \eqref{eq:s} into \eqref{eq:R2}, we have
\begin{align*}
  & (e^{2i\theta}-1)rz^4+4i((e^{2i\theta}+1)r-e^{i\theta})z^3-6(e^{2i\theta}-1)rz^2 \\
  & \qquad
     -4i((e^{2i\theta}+1)r+e^{i\theta})z +(e^{2i\theta}-1)r=0.
\end{align*}
Dividing the above equation by $ 2ie^{i\theta} $ gives,
\begin{equation}\label{eq:R3}
  r\sin \theta z^4+2(2r\cos \theta -1)z^3-6r\sin \theta z^2
    -2(2r\cos\theta -1)z+r\sin\theta =0.
\end{equation}

Note that equation \eqref{eq:R3} is an equation of degree four 
with real coefficients. 
It is known that the equation 
$$  ax^4 + bx^3 + cx^2 + dx + e = 0 \quad (a,b,c,d\in\mathbb{R}) $$
has four real roots if the following conditions are satisfied
(see, for example \cite{Rees}),
$$  \Delta>0,\qquad P<0 \quad\mbox{and}\quad D<0, $$ 
where
\begin{align*}
  \Delta=&
    256e^3a^3+(-192e^2db-128e^2c^2+144ed^2c-27d^4)a^2 \\
    & +\big((144e^2c-6ed^2)b^2+(-80edc^2+18d^3c)b+16ec^4-4d^2c^3\big)a \\
    & -27e^2b^4+(18edc-4d^3)b^3+(-4ec^3+d^2c^2)b^2,\\
  P=& 8ac-3b^2, \\
  D=& 64a^3e-16a^2c^2+16ab^2c-16a^2db-3b^4.
\end{align*}

For $ r>1 $, $ a=r\sin\theta$, $ b=2(2r\cos \theta -1)$, $ c=-6r\sin \theta$,
$d=-2(2r\cos\theta -1)$, and $ e=r\sin\theta$, we have
\begin{align*}
  \Delta & =  256((64\cos^6\theta+192\sin^2\theta\cos^4\theta
          +192\sin^4\theta\cos^2\theta+64\sin^6\theta)r^6\\
      & \quad +(-48\cos^4\theta-96\sin^2\theta\cos^2\theta-48\sin^4\theta)r^4
         +(12\cos^2\theta-15\sin^2\theta)r^2-1)\\
      & = 256(64r^6-48r^4+(-27\sin^2\theta+12)r^2-1)\\
      & >  256(64r^6-48r^4-15r^2-1)>256(64r^6-64r^4)>0, \\
  P   & = -12((4\cos^2\theta+4\sin^2\theta)r^2-4r\cos \theta +1)\\
      & =-12(4r^2-4r\cos \theta +1)
          <-12\big(4r(r-\cos\theta)+1\big)<0,\\
  D   & =-16((48\cos^4\theta+80\sin^2\theta\cos^2\theta+32\sin^4\theta)r^4
        +(-96\cos^3\theta-96\sin^2\theta\cos \theta)r^3 \\
      &\quad  +(72\cos^2\theta+28\sin^2\theta)r^2-24\cos \theta r+3)\\
      & = -16\bigg((4r^2+11)\Big(2r\cos \theta-\dfrac{6(4r^2+1)}{4r^2+11}\Big)^2
             +\dfrac{(4r^2-1)^2(8r^2-3)}{4r^2+11}\bigg) <0.
\end{align*}
Hence, \eqref{eq:R3} has four real roots, and the assertion is obtained.
\end{proof}

\section{Envelope of parabola directrices: tear drop curve }


Let $ a $ ($ a>1 $) be a point on the real axis.
Suppose light arriving from a certain direction is 
reflected at point $ w $  on the unit circle  and reaches the point $ a $.
Since the light rays cannot pass through the interior of 
the unit circle, we assume that 
$ \arg w \leq |\sin^{-1}\frac{\sqrt{a^2-1}}{a}| $.

\begin{figure}[htbp]
  \centerline{\includegraphics[width=0.5\linewidth]{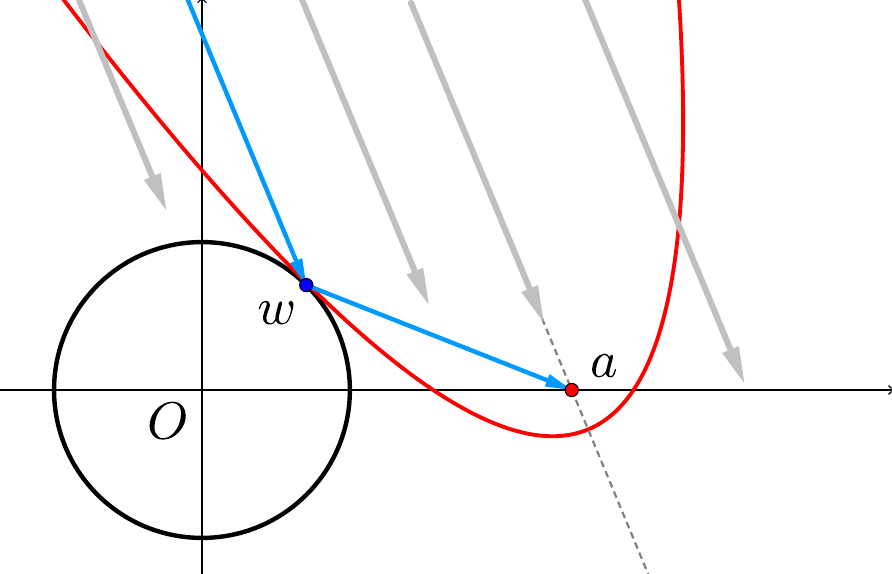}}
  \caption{A light signal is reflected at the point $w$ 
           of the unit circle  and  reaches the point $a$. 
           It comes from an infinitely distant location 
           in a specific direction.}
\end{figure}

\begin{lem}\quad
  The directrix of the parabola, tangent to the unit circle at $ w $ and with focus at $a$, is given by
  \begin{equation}\label{eq:directrix}
     (a-w)z+w^3(wa-1)\overline{z}+2w^2a^2-3w(w^2+1)a+4w^2=0.
  \end{equation}
  \end{lem}

\begin{nonsec}{\bf Remark.}
\eqref{eq:directrix} can also be written as,
 $$ 
     (a\overline{w}^2-\overline{w})z+(aw^2-w)\overline{z}+2a^2
       -3(w+\overline{w})a+4=0.
 $$
\end{nonsec}

\begin{proof}
The line $ l_w $ tangent to the unit circle at point $ w $ is given by,
\begin{equation}\label{eq:tangent}
  l_w \,:\, \ L_w(z)=z+w^2\overline{z}-2w=0.
\end{equation}
Let $ a^{\ast} $ be the symmetric point of a point $ a $
with respect to the line $ l_w $.
Therefore, the perpendicular bisector of line segment $ [a,a^{\ast}] $
coincides with $ l_w $.
Here, the perpendicular bisector of line segment $ [a,a^{\ast}] $ is given by
$ |z-a|=|z-a^{\ast}| $, i.e.
$$
    PB(z)=(a-\overline{a^{\ast}})z+(a-a^{\ast})\overline{z}
       -(a^2-|a^{\ast}|^2)=0.
$$
Since the above equation is equivalent to \eqref{eq:tangent},
and hence       
$ (a-\overline{a^{\ast}})L_w(z)-PB(z)\equiv0 $ holds.
Comparing the coefficients, we have
$$
    (-w^2+1)a-a^{\ast}+\overline{a^{\ast}}w^2=0,
     \quad\textrm{and}\quad
    -a^2+2wa+a^{\ast}\overline{a^{\ast}}-2\overline{a^{\ast}}w=0.
$$
Hence we have, $ a^{\ast}=w(2-aw)\,.$

Here, the directrix is perpendicular to the line passing through 
$ a^{\ast},  w $ and passes through $ a^{\ast} $, 
so the equation is given by
$$
   (\overline{w}-\overline{a^{\ast}})(z-a^{\ast})
    +(w-a^{\ast})(\overline{z}-\overline{a^{\ast}})=0.
$$
Substituting $ a^{\ast}=w(2-aw)$ into the above equation, we have
$$
     (a-w)z+w^3(wa-1)\overline{z}+2w^2a^2-3w(w^2+1)a+4w^2=0.
$$
\end{proof}

\begin{figure}[htbp]
  \centerline{\includegraphics[width=0.5\linewidth]{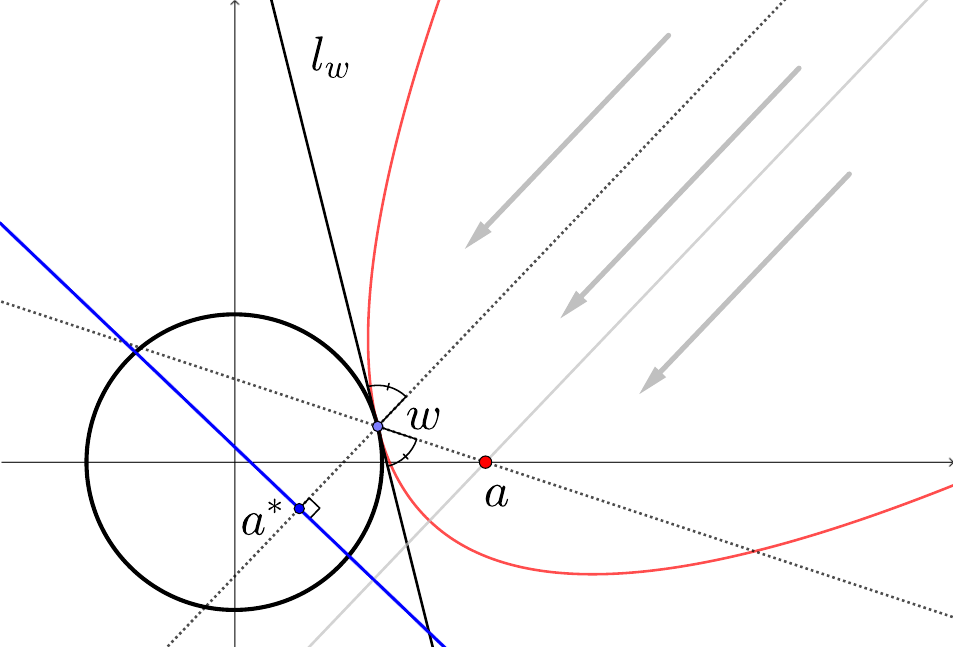}}
  \caption{The blue line indicates the directrix with respect 
           to the red parabola.}
\end{figure}

\begin{prop}
When $ a>1 $ is fixed and the direction of the light source changes, 
 that is, when $ w $ ranges over the unit circle, 
 the envelope of the family of directrices is given by the following equation
 \begin{equation}\label{eq:envelope}
   z^2\overline{z}^2-2(a^2+2)z\overline{z}+4a(z+\overline{z})+a^4-4a^2=0.
 \end{equation}
\end{prop}

\begin{proof}\quad
The envelope of the family of lines $\{l_w(z)\}_w$
is obtained by eliminating $ w $  from 
$$ 
  L_w(z)=0 \quad\mbox{and}\quad \frac{\partial}{\partial w}L_w(z)=0.
$$
Eliminating $ w $ from the above equations, we have $ e_1\cdot e_2=0$,
where
\begin{align*}
  E_1\, : \ & e_1=27a^2z^2+(-256a^6+192a^4+6a^2+4)z\overline{z}
        +27a^2\overline{z}^2\\
      & \quad
           +(288a^5-252a^3-36a)z+(288a^5-252a^3-36a)\overline{z}
           -324a^4+324a^2,\\
  E_2\,:\  &e_2= z^2\overline{z}^2+(-2a^2-4)z\overline{z}
            +4az+4a\overline{z}+a^4-4a^2.
\end{align*}
If we set $ z=x+iy $, $ E_1 $ is written as
$$
   (a^2-1)\big((8a^2+1)x-9a\big)^2+(4a^2-1)^3y^2=0.
$$
So, $ E_1 $ does not represent the envelope, and the equation of 
the envelope is given by $ E_2\,.$
\end{proof}

\begin{nonsec}{\bf Remark.}\quad
If we set $ z=x+iy $, $ E_2 $ is written as
\begin{equation}\label{eq:limacon}
    \big((x-a)^2+y^2+2a(x-a)\big)^2-4\big((x-a)^2+y^2\big)=0.
\end{equation}
Therefore, the envelope is the lima\c{c}on of Pascal \cite[Chapter 5]{l}.
The standard form of the lima\c{c}on of Pascal is
$  (x^2+y^2-px)^2-q^2(x^2+y^2)=0 $.
It is known that  this curve contains an internal loop for $ q<p $.
So, the curve defined by \eqref{eq:limacon} always contains
a loop, "a tear drop".
\end{nonsec}

Eliminating $\overline{z} $ from \eqref{eq:directrix} and \eqref{eq:envelope},
we have the following point of tangency between
the envelope and the directrix
$$
   z=2w-aw^2.
$$
Setting $ w=e^{i\theta} $ gives the parametric representation of the envelope,
$$
  E_2\,:\ z=2e^{i\theta}-ae^{2i\theta} \quad (-\pi<\theta \leq \pi).
$$
Using the expression for $ E_2 $ above, we can see that
the teardrop-shaped part of this curve corresponds to the subfamily
$ \{l_w\}_w $, where $ \theta=\arg w \leq |\sin^{-1}\frac{\sqrt{a^2-1}}{a}| $.

\begin{figure}[htbp]
\centerline{\includegraphics[width=0.45\linewidth]{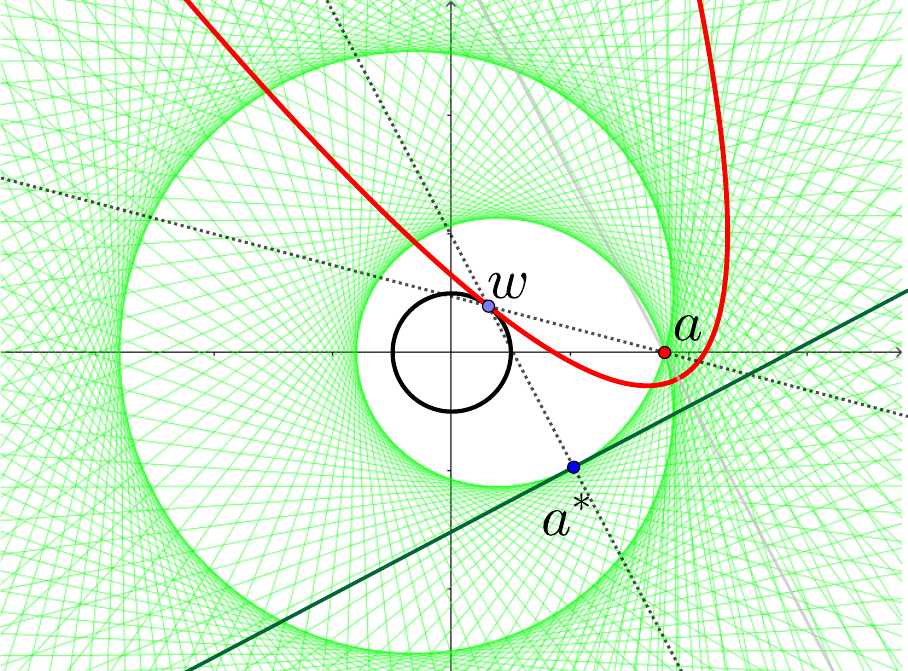}
            \quad
            \includegraphics[width=0.45\linewidth]{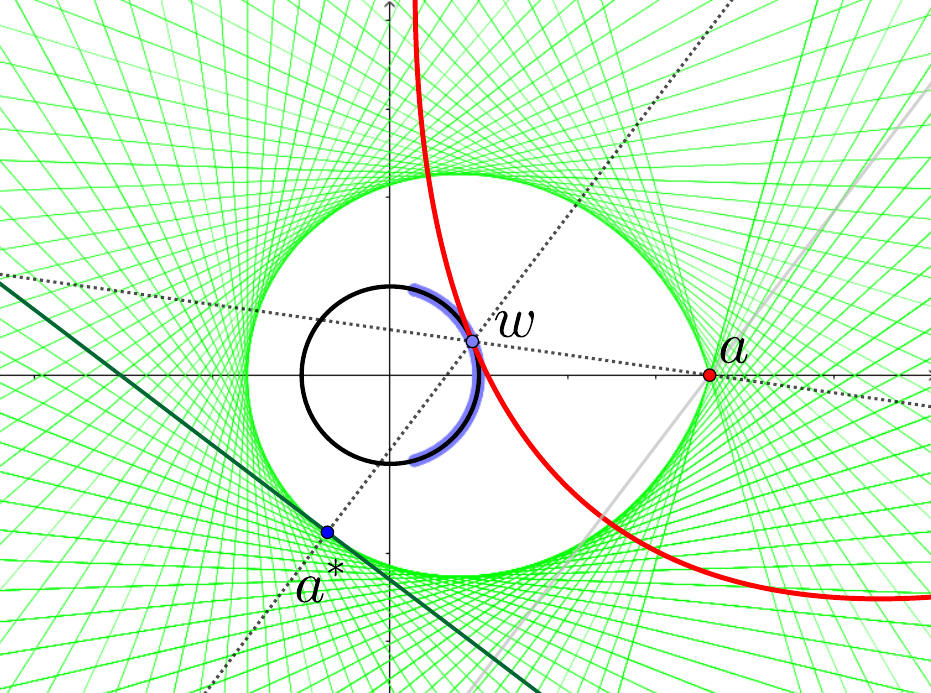}}
\caption{The lima\c{c}on curve as an envelope.}
\end{figure}

The lima\c{c}on of Pascal  can be also drawn as the pedal curve of
a circle and a fixed point $ P $.
Figure \ref{fig:lim} shows the envelope \eqref{eq:envelope}
and the lima\c{c}on curve as the pedal curve of the unit circle
and the fixed point $ a=2 $.

\begin{figure}[htbp]
\centerline{\includegraphics[width=0.5\linewidth]{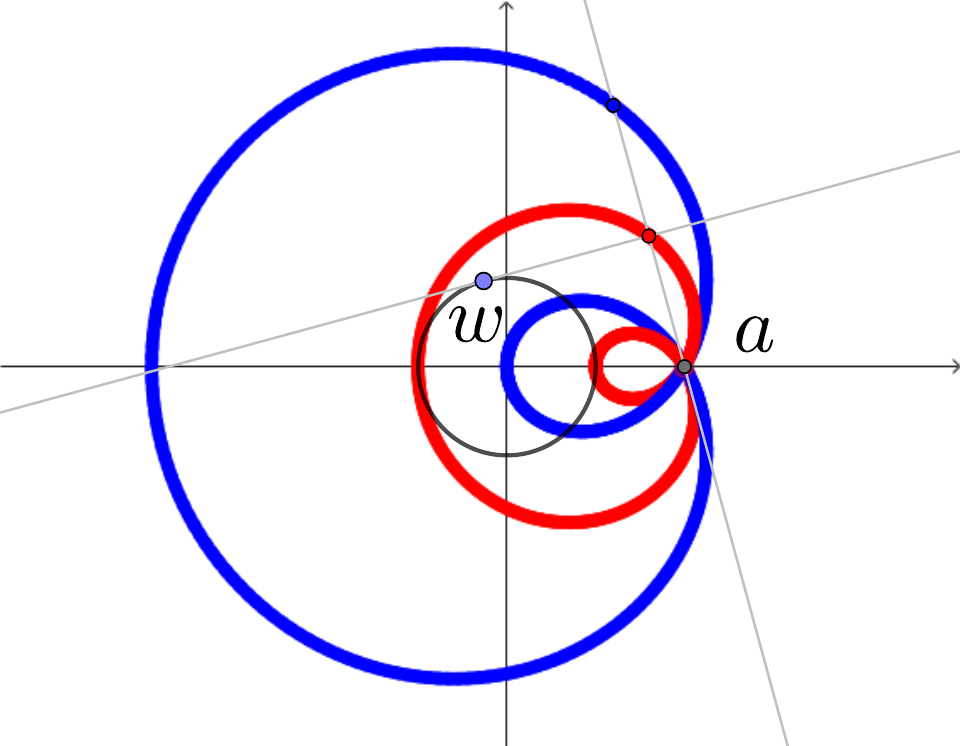}}
\caption{The lima\c{c}on curves:
  The red curve is the lima\c{c}on of Pascal as the pedal of the unit
  circle and the point $ a=2 $.
  The blue curve indicates the envelope of the directrics.}
\label{fig:lim}
\end{figure}




\subsection*{Acknowledgments.}
We are indebted to Dr. 
Vitor Hugo de Almeida Jr. for suggesting in \cite{d} this problem with source at infinity.


\medskip

\bibliographystyle{siamplain}

\end{document}